\documentclass[12pt, a4paper, leqno]{article}
\usepackage[margin=2.3cm]{geometry}
\usepackage[english]{babel}

\usepackage[runin]{abstract}
\usepackage{titling}

\usepackage{amsmath}
\usepackage{amsthm}
\usepackage{amsfonts}
\usepackage{amssymb}
\usepackage{eufrak}
\usepackage{stmaryrd}
\usepackage{bbm}
\usepackage{mathtools}
\usepackage{clrscode}
\usepackage{enumitem}
\usepackage{mdwlist}
\usepackage[backend=bibtex,style=numeric,firstinits=true,sorting=nyt,maxnames=4]{biblatex}

\abslabeldelim{.}

\newcommand{\keywordsname}{Key words}
\makeatletter
\newcommand{\keywords}[1]{%
\begin{@bstr@ctlist}
\hspace*{\abstitleskip}{\abstractnamefont\keywordsname\@bslabeldelim}\abstracttextfont\
#1%
\par\end{@bstr@ctlist}
}
\makeatother

\newcommand{\subjclassname}{Mathematics subject classification}
\makeatletter
\newcommand{\subjclass}[2][2010]{%
\begin{@bstr@ctlist}
\hspace*{\abstitleskip}{\abstractnamefont\subjclassname\ (#1)\@bslabeldelim}\abstracttextfont\
#2%
\par\end{@bstr@ctlist}
}
\makeatother

\makeatletter
\def\and{
	\end{tabular}%
	and%
	\begin{tabular}[t]{c}}%
\makeatother

\makeatletter
\def\thanks#1{
\protected@xdef\@thanks{\@thanks
\protect\footnotetext[\the\c@footnote]{#1}}%
}
\makeatother

\makeatletter
\let\addresses\@empty      
\newcommand{\address}[2][]{\g@addto@macro\addresses{\address{#1}{#2}}}
\newcommand{\curraddr}[2][]{\g@addto@macro\addresses{\curraddr{#1}{#2}}}
\newcommand{\email}[2][]{\g@addto@macro\addresses{\email{#1}{#2}}}
\newcommand{\urladdr}[2][]{\g@addto@macro\addresses{\urladdr{#1}{#2}}}
\def\enddoc@text{
  \ifx\@empty\addresses \else\@setaddresses\fi}
\AtEndDocument{\enddoc@text}
\def\emailaddrname{E-mail address}
\def\@setaddresses{\par
  \nobreak \begingroup
  \interlinepenalty\@M
  \def\address##1##2{\begingroup%
    \par\addvspace\bigskipamount
    \@ifnotempty{##1}{(\ignorespaces##1\unskip) }%
    {\noindent\ignorespaces##2}\par\endgroup}%
  \def\email##1##2{\begingroup
    \@ifnotempty{##2}{\nobreak\noindent\emailaddrname
      \@ifnotempty{##1}{, \ignorespaces##1\unskip}\/:\space
      \ttfamily##2\par}\endgroup}%
  \addresses
  \endgroup
}

\makeatother

\makeatletter
\def\cstar#1{\expandafter\@cstar\csname c@#1\endcsname}
\def\@cstar#1{\ifcase#1\or $\ast$\or $\ast\ast$\or $\ast\ast\ast$\fi}
\AddEnumerateCounter{\cstar}{\@cstar}{$\ast\ast\ast$}
\makeatother

\newlist{conditions}{enumerate}{1}
\newlist{iconditions}{enumerate}{1}
\newlist{questions}{enumerate}{1}
\setlist[conditions]{label=\normalfont(\alph*),ref=\normalfont\alph*}
\setlist[iconditions]{label=\normalfont(\roman*),ref=\normalfont\roman*}
\setlist[questions]{label=\normalfont(Q\arabic*),ref=\normalfont Q\arabic*}

\newcommand{\rank}{\func{rank}}
\newcommand{\Z}{\mathbb{Z}}
\newcommand{\R}{\mathbb{R}}
\newcommand{\C}{\mathcal{C}}
\newcommand{\NF}{\mathfrak{N}}
\newcommand{\RC}{\mathcal{R}}
\newcommand{\SB}{\mathbb{S}}
\newcommand{\codim}{\func{codim}}
\newcommand{\one}{\mathbbm{1}}
\newcommand{\Reg}{\func{Reg}}
\newcommand{\Halg}{H^{\mathrm{alg}}}

\newtheorem{theorem}{Theorem}[section]
\newtheorem{corollary}[theorem]{Corollary}
\newtheorem{proposition}[theorem]{Proposition}
\newtheorem{lemma}[theorem]{Lemma}
\theoremstyle{definition}
\newtheorem{example}[theorem]{Example}
\theoremstyle{remark}
\newcounter{assertionLetter}
\newtheorem{assertion}{\indent Assertion}[assertionLetter]

\DeclarePairedDelimiter\abs{\lvert}{\rvert}%
\DeclarePairedDelimiter\norm{\lVert}{\rVert}%

\makeatletter
\let\oldabs\abs
\def\abs{\@ifstar{\oldabs}{\oldabs*}}
\let\oldnorm\norm
\def\norm{\@ifstar{\oldnorm}{\oldnorm*}}
\makeatother

\mathchardef\mhyphen="2D

\DeclareFieldFormat{title}{#1}
\DeclareFieldFormat
  [article,inbook,incollection,inproceedings,patent,thesis,unpublished]
  {title}{#1\isdot}
\DeclareFieldFormat{journaltitle}{#1}
\DeclareFieldFormat{booktitle}{#1}

\renewbibmacro*{in:}{%
  }

\DeclareFieldFormat[article]{pages}{#1}

\DeclareFieldFormat[article,periodical]{series}{
  \ifinteger{#1}
    {\mkbibparens{#1}}
    {\ifbibstring{#1}{\bibstring{#1}}{#1}}}

\DeclareBibliographyDriver{book}{%
  \usebibmacro{bibindex}%
  \usebibmacro{begentry}%
  \usebibmacro{author/editor+others/translator+others}%
  \setunit{\labelnamepunct}\newblock
  \usebibmacro{maintitle+title}%
  \newunit
  \printlist{language}%
  \newunit\newblock
  \usebibmacro{byauthor}%
  \newunit\newblock
  \usebibmacro{byeditor+others}%
  \newunit\newblock
  \printfield{edition}%
  \newunit\newblock
  \usebibmacro{series+number}%
  \iffieldundef{maintitle}
    {\printfield{volume}%
     \printfield{part}}
    {}%
  \newunit
  \printfield{volumes}%
  \newunit\newblock
  \printfield{note}%
  \newunit\newblock
  \usebibmacro{publisher+location+date}%
  \newunit\newblock
  \usebibmacro{chapter+pages}%
  \newunit
  \printfield{pagetotal}%
  \newunit\newblock
  \iftoggle{bbx:isbn}
    {\printfield{isbn}}
    {}%
  \newunit\newblock
  \usebibmacro{doi+eprint+url}%
  \newunit\newblock
  \usebibmacro{addendum+pubstate}%
  \setunit{\bibpagerefpunct}\newblock
  \usebibmacro{pageref}%
  \usebibmacro{finentry}}

\renewbibmacro*{publisher+location+date}{%
  \printlist{publisher}%
  \setunit*{\addcomma\space}%
  \printlist{location}%
  \setunit*{\addcomma\space}%
  \usebibmacro{date}%
  \newunit}

\DeclareBibliographyDriver{online}{%
  \usebibmacro{bibindex}%
  \usebibmacro{begentry}%
  \usebibmacro{author/editor+others/translator+others}%
  \setunit{\labelnamepunct}\newblock
  \usebibmacro{title}%
  \newunit
  \printlist{language}%
  \newunit\newblock
  \usebibmacro{byauthor}%
  \newunit\newblock
  \usebibmacro{byeditor+others}%
  \newunit\newblock
  \printfield{version}%
  \newunit
  \printfield{note}%
  \newunit\newblock
  \printlist{organization}%
  \newunit\newblock
  \iftoggle{bbx:eprint}
    {\usebibmacro{eprint}}
    {}%
  \newunit\newblock
  \usebibmacro{date}%
  \newunit\newblock
  \usebibmacro{url+urldate}%
  \newunit\newblock
  \usebibmacro{addendum+pubstate}%
  \setunit{\bibpagerefpunct}\newblock
  \usebibmacro{pageref}%
  \usebibmacro{finentry}}

\makeatletter
\DeclareFieldFormat{eprint:arxiv}{%
  arXiv\addcolon\linebreak[2]%
  \ifhyperref
    {\href{http://arxiv.org/\abx@arxivpath/#1}{%
       \nolinkurl{#1}%
       \iffieldundef{eprintclass}
	 {}
	 {\addspace\texttt{\mkbibbrackets{\thefield{eprintclass}}}}}}
    {\nolinkurl{#1}
     \iffieldundef{eprintclass}
       {}
       {\addspace\texttt{\mkbibbrackets{\thefield{eprintclass}}}}}}
\makeatother

\DeclareBibliographyDriver{inproceedings}{%
  \usebibmacro{bibindex}%
  \usebibmacro{begentry}%
  \usebibmacro{author/translator+others}%
  \setunit{\labelnamepunct}\newblock
  \usebibmacro{title}%
  \newunit
  \printlist{language}%
  \newunit\newblock
  \usebibmacro{byauthor}%
  \newunit\newblock
  \usebibmacro{in:}%
  \usebibmacro{maintitle+booktitle}%
  \newunit\newblock
  \usebibmacro{event+venue+date}%
  \newunit\newblock
  \usebibmacro{byeditor+others}%
  \newunit\newblock
  \iffieldundef{maintitle}
    {\printfield{volume}%
     \printfield{part}}
    {}%
  \newunit
  \printfield{volumes}%
  \newunit\newblock
  \usebibmacro{series+number}%
  \newunit\newblock
  \printfield{note}%
  \newunit\newblock
  \printlist{organization}%
  \newunit
  \usebibmacro{chapter+pages}%
  \newunit\newblock
  \usebibmacro{publisher+location+date}%
  \newunit\newblock
  \iftoggle{bbx:isbn}
    {\printfield{isbn}}
    {}%
  \newunit\newblock
  \usebibmacro{doi+eprint+url}%
  \newunit\newblock
  \usebibmacro{addendum+pubstate}%
  \setunit{\bibpagerefpunct}\newblock
  \usebibmacro{pageref}%
  \usebibmacro{finentry}}

\title{\bf Continuous rational maps into spheres}
\date{}
\author{Wojciech Kucharz\thanks{The author was partially supported by
NCN grant 2011/01/B/ST1/01289.}}

\address{Institute of Mathematics\\Faculty of Mathematics and Computer
Science\\Jagiellonian University\\\L{}ojasiewicza 6\\30-348
Krak\'ow\\Poland}
\email{Wojciech.Kucharz@im.uj.edu.pl}

\usepackage[pdftex, pdfauthor={\theauthor}, pdftitle={\thetitle}]{hyperref}

\begin{document}
\maketitle
\thispagestyle{empty}

\begin{abstract}
Let $X$ be a compact nonsingular real algebraic variety. We prove that
if a continuous map from $X$ into the unit $p$-sphere is homotopic to a
continuous rational map, then, under certain assumptions, it can be
approximated in the compact-open topology by continuous rational maps.
As a byproduct, we also obtain some results on approximation of smooth
submanifolds by nonsingular subvarieties.
\end{abstract}



\section{Introduction and main results}\label{sec-1}

Throughout this paper the term \emph{real algebraic variety} designates
a locally ringed space isomorphic to an algebraic subset of $\R^n$, for
some $n$, endowed with the Zariski topology and the sheaf of real-valued
regular functions (such an object is called an affine real algebraic
variety in \cite{bib2}). Nonsingular varieties are assumed to be of
pure dimension. The class of real algebraic varieties is identical with
the class of quasi-projective real varieties, cf.
\cite[Proposition~3.2.10, Theorem~3.4.4]{bib2}. Morphisms of real
algebraic varieties are called \emph{regular maps}. Each real algebraic
variety carries also the Euclidean topology, which is induced by the
usual metric on $\R$. Unless explicitly stated otherwise, all
topological notions relating to real algebraic varieties refer to the
Euclidean topology.

Let $X$ and $Y$ be real algebraic varieties. A map $f \colon X \to Y$ is
said to be \emph{continuous rational} if it is continuous and there
exists a Zariski open and dense subvariety $U$ of $X$ such that the
restriction $f|_U \colon U \to Y$ is a regular map. Let $X(f)$ denote
the union of all such $U$. The complement $P(f) = X \setminus X(f)$ of
$X(f)$ is called the \emph{irregularity locus} of $f$. Thus $P(f)$ is
the smallest Zariski closed subvariety of $X$ for which the restriction
$f|_{X \setminus P(f)} \colon X \setminus P(f) \to Y$ is a regular map.
If $f(P(f)) \neq Y$, we say that $f$ is a \emph{nice} map. There exist
continuous rational maps that are not nice, cf.
\cite[Example~2.2~(ii)]{bib16}. Continuous rational maps have only
recently become the object of serious investigation, cf. \cite{bib9,
bib15, bib16, bib17, bib18}.

The space $\C(X,Y)$ of all continuous maps from $X$ to $Y$ will always
be endowed with the compact-open topology. There are the following
inclusions
\begin{equation*}
\C(X,Y) \supseteq \RC^0(X,Y) \supseteq \RC_0(X,Y) \supseteq \RC(X,Y),
\end{equation*}
where $\RC^0(X,Y)$ is the set of all continuous rational maps,
$\RC_0(X,Y)$ consists of the nice maps in $\RC^0(X,Y)$, and $\RC(X,Y)$
is the set of regular maps. By definition, a continuous map from $X$ into
$Y$ can be approximated by continuous rational maps if it belongs to the
closure of $\RC^0(X,Y)$ in $\C(X,Y)$. Approximation by nice continuous
rational maps or regular maps is defined in the analogous way.

Henceforth we assume that the variety $X$ is compact and nonsingular,
and concentrate our attention on maps with values in the unit $p$-sphere
\begin{equation*}
\SB^p = \{ (u_0, \ldots, u_p) \in \R^{p+1} \mid u_0^2 + \cdots + u_p^2 =1
\}.
\end{equation*}
Any continuous map $h \colon X \to \SB^p$ has a neighborhood in $\C(X,
\SB^p)$ consisting entirely of maps homotopic to $h$. The following two
natural questions are of interest:

\begin{questions}%
\item\label{Q1} If $h$ is homotopic to a regular map, can it be approximated by
regular maps?

\item\label{Q2} If $h$ is homotopic to a continuous rational map, can it be
approximated by continuous rational maps?
\end{questions}

If $\dim X < p$, then the answer to either of these questions is ``yes''
since $\R^p$ is biregularly isomorphic to $\SB^p$ with one point
removed.

Assume then that $\dim X \geq p$. The answer to (\ref{Q1}) is affirmative
for $p \in \{ 1,2,4 \}$, cf. \cite{bib3} or \cite{bib2}. The answer to
(\ref{Q2}) is affirmative for $\dim X \leq p+1$ or $p \in \{1,2,4 \}$,
cf. \cite{bib17}. Nothing is known about (\ref{Q1}) and (\ref{Q2}) in
other cases. Actually, the sets $\RC(X, \SB^1)$ and $\RC^0(X, \SB^1)$
have equal closures in $\C(X, \SB^1)$, cf. \cite{bib16}. Furthermore,
the set $\RC_0(X, \SB^p)$ is dense in $\C(X, \SB^p)$ if $\dim X =p$,
cf. \cite{bib17}. On the other hand,
the closure of $\RC(\SB^1 \times \SB^1, \SB^2)$ in $\C(\SB^1 \times
\SB^1, \SB^2)$ coincides with the set of all continuous null homotopic
maps, and hence is different from $\C(\SB^1 \times \SB^1, \SB^2)$, cf.
\cites[Theorem~2.4]{bib4}[Proposition~2.2]{bib6}. If $\dim X = p + 1$,
then the closures of $\RC_0(X, \SB^p)$ and $\RC^0(X, \SB^p)$ in $\C(X,
\SB^p)$ are identical, cf. \cite{bib17}. No continuous rational
map is known that is not homotopic to a nice continuous rational map.
Similarly, no continuous rational map is known that cannot be
approximated by nice continuous rational maps. In
the present paper we obtain new results, related to (\ref{Q1}) and
(\ref{Q2}), on nice continuous rational maps. All results announced in
this section are proved in Section~\ref{sec-2}.

\begin{theorem}\label{th-1-1}
Let $X$ be a compact nonsingular real algebraic variety and let $p$ be
an integer satisfying $\dim X + 3 \leq 2p$. For a continuous map $h
\colon X \to \SB^p$, the following conditions are equivalent:
\begin{conditions}
\item\label{th-1-1-a} $h$ can be approximated by nice continuous
rational maps.

\item\label{th-1-1-b} $h$ is homotopic to a nice continuous rational map.
\end{conditions}
\end{theorem}

Our next result requires some preparation.

For any $k$-dimensional
compact smooth (of class $\C^{\infty}$) manifold $K$, let $[K]$ denote
its fundamental class in the homology group $H_k(K; \Z/2)$. If $T$ is a
topological space and $K$ is a subspace of $T$, we denote by $[K]_T$ the
homology class in $H_k(T;\Z/2)$ represented by $K$, that is, $[K]_T =
i_*([K])$, where $i \colon K \hookrightarrow T$ is the inclusion map.

Let $X$ be a nonsingular real algebraic variety. We denote by $A_k(X)$
the subgroup of $H_k(X; \Z/2)$ generated by all homology classes of the
form $[Z]_X$, where $Z$ is a $k$-dimensional nonsingular Zariski locally
closed subvariety of $X$ that is compact and orientable as a smooth
manifold. Here $Z$ is not assumed to be Zariski closed in $X$.

For any positive integer $p$, let $s_p$ be the unique generator of the
cohomology group $H^p(\SB^p; \Z/2) \cong \Z/2$.

Recall that a smooth manifold is said to be \emph{spin} if it is
orientable and its second Stiefel--Whitney class vanishes.

\begin{theorem}\label{th-1-2}
Let $X$ be a compact nonsingular real algebraic variety of dimension
$p+2$, where $p \geq 5$. Assume that $X$ is a spin manifold. For a
continuous map $h \colon X \to \SB^p$, the following conditions are
equivalent:
\begin{conditions}
\item\label{th-1-2-a} $h$ can be approximated by nice continuous rational
maps.

\item\label{th-1-2-b} $h$ is homotopic to a nice continuous rational
map.

\item\label{th-1-2-c} The homology class Poincar\'e dual to the
cohomology class $h^*(s_p)$ belongs to $A_2(X)$.
\end{conditions}
\end{theorem}

Denote by
\begin{equation*}
\rho \colon H_*(-; \Z) \to H_*(-; \Z/2)
\end{equation*}
the reduction modulo $2$ homomorphism.

\begin{corollary}\label{cor-1-3}
Let $X$ be a compact nonsingular real algebraic variety of dimension
${p+2}$, where $p \geq 5$. Assume that $X$ is a spin manifold. Then the
following conditions are equivalent:
\begin{conditions}
\item\label{cor-1-3-a} Each continuous map from $X$ into $\SB^p$ can be
approximated by nice continuous rational maps.

\item\label{cor-1-3-b} Each continuous map from $X$ into $\SB^p$ is
homotopic to a nice continuous rational map.

\item\label{cor-1-3-c} $A_2(X) = \rho(H_2(X; \Z))$.
\end{conditions}
\end{corollary}

In some cases, condition~(\ref{cor-1-3-c}) in Corollary~\ref{cor-1-3}
can easily be verified.

\begin{example}\label{ex-1-4}
Let $X = C_1 \times \cdots \times C_{p+2}$, where each $C_i$ is a
compact connected nonsingular real algebraic curve, $1 \leq i \leq p+2$.
If $p \geq 5$, then the set $\RC_0(X, \SB^p)$ of nice continuous
rational maps is dense in $\C(X, \SB^p)$. Indeed, $A_2(X) = H_2(X;
\Z/2)$ and hence the assertion follows from Corollary~\ref{cor-1-3}.
\end{example}

If $X$ is as in Example~\ref{ex-1-4}, then in general there exist
continuous maps from $X$ into $\SB^p$ that cannot be approximated by
regular maps, cf.~\cite{bib6}.

According to \cite{bib17}, for any positive integers $n$ and $p$, the
set $\RC_0(\SB^n, \SB^p)$ of nice continuous rational maps is dense in
$\C(\SB^n, \SB^p)$. In this paper we obtain other density results. As in
\cite{bib2, bibb}, given a compact nonsingular real algebraic variety
$X$, we denote by $\Halg_k(X; \Z/2)$ the subgroup of $H_k(X; \Z/2)$
generated by all homology classes represented by $k$-dimensional Zariski
closed (possibly singular) subvarieties of $X$. It easily follows that
\begin{equation*}
A_k(X) \subseteq \Halg_k(X; \Z/2).
\end{equation*}
If $k \leq d$ and
\begin{equation*}
\Halg_k(X; \Z/2) = H_k(X; \Z/2)
\end{equation*}
then the K\"unneth formula implies that
\begin{equation*}
\Halg_k(X \times \SB^d; \Z/2) = H_k(X \times \SB^d; \Z/2).
\end{equation*}
Conversely, the latter equality implies the former (with no restriction
on $k$ and $d$) since
\begin{equation*}
\pi_* (H_k(X \times \SB^d; \Z/2)) = H_k(X; \Z/2),
\end{equation*}
where $\pi \colon X \times \SB^d \to X$ is the canonical projecion, and
$\Halg_k(-; \Z/2)$ is functorial for regular maps between compact real
algebraic varieties, cf. \cite[{}5.12]{bib8} or \cite[p.~53]{biba}.

\begin{theorem}\label{th-1-5}
Let $X$ be a compact nonsingular real algebraic variety of dimension
$n$. Let $d$ and $p$ be positive integers satisfying $n+1 \leq p$ and
$n+2d+1 \leq 2p$. If
\begin{equation*}
\Halg_i(X; \Z/2) = H_i(X; \Z/2)
\end{equation*}
for every integer $i$ with $0 \leq i \leq n+d-p$, then the set $\RC_0(X
\times \SB^d, \SB^p)$ of nice continuous rational maps is dense in $\C(X
\times \SB^d, \SB^p)$.
\end{theorem}

It is worthwhile to record the following observation.

\begin{example}\label{ex-1-6}
Let $d$, $n$ and $p$ be positive integers satisfying one of the following
two conditions:
\begin{iconditions}
\item $n + d \geq 7$ and $p = n + d - 2$;
\item $n + 1 \leq p$ and $n + 2d + 1 \leq 2p$.
\end{iconditions}
In view of Corollary~\ref{cor-1-3} and Theorem~\ref{th-1-5}, the set
$\RC_0(\SB^n \times \SB^d, \SB^p)$ of nice continuous rational maps is
dense in $\C(\SB^n \times \SB^d, \SB^p)$. Furthermore, by
\cite[Corollary~1.3, Theorem~1.7]{bib17}, ${\RC_0(\SB^n \times \SB^d,
\SB^p)}$ is dense in $\C(\SB^n \times \SB^d, \SB^p)$ if $n+d-p \leq 1$. It
would be interesting to decide whether or not this density assertion
holds with no restrictions on $d$, $n$ and $p$.
\end{example}

We have one more density result.

\begin{theorem}\label{th-1-7}
Let $X$ be a compact nonsingular real algebraic variety of dimension
$p$, where $p \geq 5$. If
\begin{equation*}
\Halg_2(X; \Z/2) = H_2(X; \Z/2)
\end{equation*}
and $X$ is a spin manifold, then the set $\RC_0(X \times \SB^2, \SB^p)$
of nice continuous rational maps is dense in $\C(X \times \SB^2, \SB^p)$.
\end{theorem}

In topology, one often achieves stabilization effects by making use of
the suspension. For problems involving continuous rational maps, the
following construction can serve as a substitute for the suspension.

For any positive integer $p$, let
\begin{equation*}
\sigma_p \colon \SB^p \times \SB^1 \to \SB^{p+1}
\end{equation*}
be a continuous map such that for some nonempty open subset $U_p$ of
$\SB^{p+1}$, the restriction $\sigma_{U_p} \colon \sigma_p^{-1}(U_p) \to
U_p$ of $\sigma_p$ is a smooth diffeomorphism. We denote by $\one$ the
identity map of~$\SB^1$.

\begin{theorem}\label{th-1-8}
Let $X$ be a compact nonsingular real algebraic variety and let $p$ be a
positive integer. For any continuous map $h \colon X \to \SB^p$, the
following conditions are equivalent:
\begin{conditions}
\item\label{th-1-8-a} The map $\sigma_p \circ (h \times \one) \colon X
\times \SB^1 \to \SB^{p+1}$ can be approximated by nice continuous
rational maps.

\item\label{th-1-8-b} The map $\sigma_p \circ (h \times \one) \colon X
\times \SB^1 \to \SB^{p+1}$ is homotopic to a nice continuous rational
map.
\end{conditions}
\end{theorem}

In Section~\ref{sec-2} we derive the results stated above from certain
results concerning approximation of smooth submanifolds by nonsingular
subvarieties. Approximation of smooth submanifolds, being of independent
interest, is further investigated in Section~\ref{sec-3}.

\section{\texorpdfstring{Weak algebraic approximation of smooth\\
submanifolds}{Weak algebraic approximation of smooth submanifolds}}\label{sec-2}
For any smooth manifolds (with possibly nonempty boundary) $N$ and $P$,
let $\C^{\infty}(N,P)$ denote the space of all smooth maps from $N$ into
$P$ endowed with the $\C^{\infty}$ topology, cf.~\cite{bib11}. The source
manifold will always be assumed to be compact, and hence the weak
$\C^{\infty}$ topology coincides with the strong one.

Let $X$ be a nonsingular real algebraic variety. A compact smooth
submanifold $M$ of $X$ is said to \emph{admit a weak algebraic
approximation in $X$} if each neighborhood of the inclusion map $M
\hookrightarrow X$ in the space $\C^{\infty}(M,X)$ contains a smooth
embedding $e \colon M \to X$ such that $e(M)$ is a nonsingular Zariski
locally closed subvariety of $X$. If $e$ can be chosen so that $e(M)$ is
a nonsingular Zariski closed subvariety of $X$, then $M$ is said to
\emph{admit an algebraic approximation in $X$}.

Weak algebraic approximation will be essential for the proofs of
Theorems~\ref{th-1-1}, \ref{th-1-2}, \ref{th-1-5}, \ref{th-1-7} and~\ref{th-1-8}. It is also of
independent interest, cf.~\cite[Theorems~A and~F]{bib1}. In order to avoid
unnecessary restrictions, we do not assume that the ambient variety $X$
is compact. Our criteria for weak algebraic approximation are presented
in Propositions~\ref{prop-2-3} and~\ref{prop-2-6}.

For any real algebraic variety $V$, let $\Reg(V)$ denote its locus of
nonsingular points in dimension $\dim V$.

\begin{lemma}\label{lem-2-1}
Let $X$ be a nonsingular real algebraic variety and let $M$ be a compact
smooth submanifold of $X$. Assume that there exists a Zariski closed
subvariety $A$ of $X$ such that $M \cap \Reg(A) = \varnothing$ and
\begin{equation*}
M \cup \Reg(A) = \partial P,
\end{equation*}
where $P$ is a compact smooth manifold with boundary $\partial P$,
embedded in $X$ with trivial normal bundle and satisfying $P \cap A =
\Reg(A)$. If $S = A \setminus \Reg(A)$, then $M \subseteq X \setminus S$
and $M$ admits an algebraic approximation in $X\setminus S$. In
particular, $M$ admits a weak algebraic approximation in $X$.
\end{lemma}

\begin{proof}
According to Hironaka's theorem on resolution of singularities
\cite{bib10} (cf. also \cite{bib14} for a~very readable exposition), we
can assume that $X$ is a Zariski open subvariety of a~compact
nonsingular real algebraic variety $X'$. Note that either $A =
\varnothing$ or $\Reg(A)$ is a~compact smooth submanifold of $X$. If
$A'$ is a Zariski closure of $A$ in $X'$, then $A = A' \cap X$ and
$\Reg(A) = \Reg(A')$. Consequently, we can assume without loss of
generality that $X$ is compact. Then it follows directly from the proof
of Proposition~2.7 in \cite{bib17} that $M$ admits an algebraic
approximation in $X \setminus S$. Hence, $M$ admits a weak algebraic
approximation in~$X$. 
\end{proof}

\begin{lemma}\label{lem-2-2}
Let $X$ be a nonsingular real algebraic variety and let $M$ be a compact
smooth submanifold of $X$. Assume that there exists a nonsingular
Zariski locally closed subvariety~$Z$ of $X$ such that $M \cap Z =
\varnothing$ and
\begin{equation*}
M \cup Z = \partial Q,
\end{equation*}
where $Q$ is a compact smooth manifold with boundary $\partial Q$,
embedded in $X$ with trivial normal bundle. If $2 \dim M + 1 \leq \dim
X$, then $M$ admits a weak algebraic approximation in $X$.
\end{lemma}

\begin{proof}
Note that $Z$ is a compact smooth submanifold of $X$. Either $Z =
\varnothing$ or $\dim Z = \dim M$. If $A$ is the Zariski closure of $Z$
in $X$, then
\begin{equation*}
Z = \Reg(A).
\end{equation*}
Furthermore, $S \coloneqq A \setminus Z$ is a Zariski closed subvariety
of $X$ with $\dim S < \dim M$. In particular, $S$ has a finite
stratification into smooth submanifolds of $X$ of dimension at most
$\dim S$.

Assume that $2 \dim M + 1 \leq \dim X$ and let $f \colon Q
\hookrightarrow X$ be the inclusion map. Since $\dim Q = \dim M + 1$, we
get
\begin{equation*}
\dim Q + \dim S < \dim X.
\end{equation*}
In view of the transversality theorem, there exists a smooth map $g
\colon Q \to X$, arbitrarily close to $f$ in the space
$\C^{\infty}(Q,X)$, such that
\begin{equation*}
g|_{\Reg(A)} = f|_{\Reg(A)}\quad\textrm{and}\quad g(Q) \cap S =
\varnothing.
\end{equation*}
If $g$ is sufficiently close to $f$, then it is a smooth embedding
isotopic to $f$. In particular, $P \coloneqq g(Q)$ is a compact smooth
manifold with boundary
\begin{equation*}
\partial P = g(M) \cup \Reg(A),
\end{equation*}
embedded in $X$ with trivial normal bundle. By construction,
\begin{equation*}
g(M) \cap \Reg(A) = \varnothing \quad \textrm{and} \quad P \cap A =
\Reg(A).
\end{equation*}
Hence, according to Lemma~\ref{lem-2-1}, the smooth submanifold $g(M)$
of $X$ admits a weak algebraic approximation in $X$. Consequently, $M$
admits a weak algebraic approximation in $X$.
\end{proof}

\begin{proposition}\label{prop-2-3}
Let $X$ be a nonsingular real algebraic variety and let $M$ be a compact
smooth submanifold of $X$. Assume that there exists a nonsingular
Zariski locally closed subvariety $Z$ of $X$ such that
\begin{equation*}
(M \times \{0\}) \cup (Z \times \{1\}) = \partial B,
\end{equation*}
where $B$ is a compact smooth manifold with boundary $\partial B$,
embedded in $X \times \R$ with trivial normal bundle. If $2 \dim M + 3
\leq \dim X$, then $M$ admits a weak algebraic approximation in~$X$.
\end{proposition}

\begin{proof}
Note that $Z$ is a compact smooth submanifold of $X$. Either $Z =
\varnothing$ or $\dim Z = \dim M$. Assume that $2 \dim M + 3 \leq \dim
X$. We can find a small smooth isotopy which transforms $M$ onto a
smooth submanifold $M'$ of $X$ with $M' \cap Z = \varnothing$. Thus
there exists a smooth diffeotopy $\varphi \colon X \times \R \to X$ such
that $\varphi_0$ is the identity map and $\varphi_1(M) = M'$. Here, as
usual, $\varphi_t(x) = \varphi(x,t)$ for $t$ in $\R$ and $x$ in $X$. The
map
\begin{equation*}
\Phi \colon X \times \R \to X \times \R, \quad \Phi(x,t) = ( \varphi(x,
1-t), t)
\end{equation*}
is a smooth diffeomorphism. If $B' = \Phi(B)$, then
\begin{equation*}
(M' \times \{0\}) \cup (Z \times \{1\}) = \partial B'.
\end{equation*}
Therefore, replacing $M$ by $M'$ and $B$ by $B'$, we can assume that
\begin{equation*}
M \cap Z = \varnothing.
\end{equation*}

Let $f = (f_1, f_2) \colon B \hookrightarrow X \times \R$ be the
inclusion map. Since $2 \dim B + 1 \leq \dim X$, the smooth map $f_1
\colon B \to X$ can be approximated in the space $\C^{\infty}(B,X)$ by a
smooth embedding $g_1 \colon B \to X$. Furthermore, $g_1$ can be chosen
so that $g_1|_{\partial B} = f_1|_{\partial B}$. Note that
\begin{align*}
g_1(x,0) &= f_1(x,0) = x \quad \textrm{for all $x$ in $M$},\\
g_1(x,1) &= f_1(x,1) = x \quad \textrm{for all $x$ in $Z$}.
\end{align*}
By construction, $Q \coloneqq g_1(B)$ is a compact smooth submanifold of
$X$ with boundary
\begin{equation*}
\partial Q = M \cup Z.
\end{equation*}
In view of Lemma~\ref{lem-2-2}, it suffices to prove that the normal
bundle $\nu$ to $Q$ in $X$ is trivial. This can be done as follows. The
smooth embedding
\begin{equation*}
g \colon B \to X \times \R, \quad g(x,t) = (g_1(x,t), 0)
\end{equation*}
is homotopic to $f$. Since $2 \dim Q + 2 \leq \dim (X \times \R)$, the
smooth embeddings $f$ and $g$ are isotopic, cf.~\cite[Theorem~6]{bib21}
or \cite[p.~183, Exercise~10]{bib11}. Consequently, the normal bundle to
$g(B)$ in $X \times \R$ is trivial, the normal bundle to $f(B) = B$ in $X
\times \R$ being trivial. Since $g(B) = Q \times \{0\}$, it follows that
the normal bundle $\nu$ is stably trivial. Now,
\begin{equation*}
\rank \nu = \dim X - \dim Q \geq \dim Q + 1,
\end{equation*}
and hence $\nu$ is trivial, cf.~\cite[p.~100]{bib13}.
\end{proof}

\begin{proof}[Proof of Theorem~\ref{th-1-1}]
It suffices to prove that (\ref{th-1-1-b}) implies (\ref{th-1-1-a}).
Suppose that (\ref{th-1-1-b}) holds. We can assume that the map $h$ is
smooth. By Sard's theorem, $h$ is transverse to some point $y$ in
$\SB^p$. Then $M \coloneqq h^{-1}(y)$ is a compact smooth submanifold
of $X$. According to \cite[Theorem~2.4]{bib16}, there exists a
nonsingular Zariski locally closed subvariety $Z$ of $X$ such that
\begin{equation*}
(M \times \{0\}) \cup (Z \times \{1\}) = \partial B,
\end{equation*}
where $B$ is a compact smooth manifold with boundary $\partial B$,
embedded in $X \times \R$ with trivial normal bundle. In view of
Proposition~\ref{prop-2-3}, $M$ admits a weak algebraic approximation in
$X$, which in turn implies that $h$ can be approximated by nice continuous
rational maps, cf.~\cite[Theorem~1.2]{bib17}.
\end{proof}

The proof of Theorem~\ref{th-1-2} requires more preparation.

\begin{lemma}\label{lem-2-4}
Let $N$ be a smooth spin manifold. Let $P$ be a compact orientable
smooth submanifold of $N$, with possibly nonempty boundary. Assume that
$2 \dim P + 1 \leq \dim N$ and $\dim P \leq 3$. Then the normal bundle
to $P$ in $N$ is trivial.
\end{lemma}

\begin{proof}
For any smooth manifold $M$, let $\tau_M$ denote its tangent bundle. The
restriction $\tau_N|_P$ is isomorphic to the direct sum $\tau_P \oplus
\nu$. Since $P$ is orientable and $\dim P \leq 3$, it follows that $P$
is a spin manifold. Consequently, the $i$th Stiefel--Whitney class of
$\nu$ is equal to zero for $i = 1, 2$. Denote by $D$ the double of $P$
and regard $P$ as a submanifold of $D$. If $r \colon D \to P$ is the
standard retraction, then the $i$th Stiefel--Whitney class of the
pullback vector bundle $r^*\nu$ on $D$ is zero for $i = 1, 2$. This
implies that $r^*\nu$ is stably trivial, cf. \cite[Lemma~1.2]{bib5}.
Actually, $r^*\nu$ is trivial since $\rank r^*\nu > \dim D$, cf.
\cite[p.~100]{bib13}. Hence the vector bundle $\nu$ is trivial, being
isomorphic to $(r^*\nu)|_P$.
\end{proof}

For any $k$-dimensional compact oriented smooth manifold $K$, let
$\llbracket K \rrbracket$ denote its fundamental class in the homology
group $H_k(K; \Z)$. If $T$ is a topological space and $K$ is a subspace
of $T$, we denote by $\llbracket K \rrbracket_T$ the homology class in
$H_k(T; \Z)$ represented by $K$, that is, $\llbracket K \rrbracket_T =
i_*(\llbracket K \rrbracket)$, where $i \colon K \hookrightarrow T$ is
the inclusion map.

Let $X$ be a nonsingular real algebraic variety. We say that a homology
class $u$ in $H_k(X; \Z)$ is \emph{A-distinguished} if it is of the
form
\begin{equation*}
u = \llbracket Z \rrbracket_X,
\end{equation*}
where $Z$ is a $k$-dimensional nonsingular Zariski locally closed
subvariety of $X$ that is compact and oriented as a smooth manifold.

Recall that $\rho \colon H_*(-; \Z) \to H_*(-; \Z/2)$ denotes the
reduction modulo $2$ homomorphism.

\begin{lemma}\label{lem-2-5}
Let~$X$ be a nonsingular real algebraic variety of dimension at least
$5$. Assume that $X$ is a spin manifold. For a homology class $u$ in
$H_2(X; \Z)$, the following conditions are equivalent:
\begin{conditions}
\item\label{lem-2-5-a} $u$ is A-distinguished.
\item\label{lem-2-5-b} $\rho(u)$ belongs to $A_2(X)$.
\end{conditions}
\end{lemma}

\begin{proof}
We first prove two preliminary facts.

\begin{assertion}\label{a1}
If $v_1$ and $v_2$ are A-distinguished homology classes in $H_2(X;
\Z)$, then their sum $v_1 + v_2$ is A-distinguished too.
\end{assertion}

By assumption, $v_i \coloneqq \llbracket Z_i \rrbracket_X$, where $Z_i$
is a $2$-dimensional nonsingular Zariski locally closed subvariety of
$X$ that is compact and oriented as a smooth manifold, $i = 1, 2$. Let
$A_i$ be the Zariski closure of $Z_i$ in $X$. Then
\begin{equation*}
Z_i = \Reg(A_i)
\end{equation*}
and $\dim A_i = 2$. Furthermore, $S_i \coloneqq A_i \setminus Z_i$ is a
Zariski closed subvariety of $X$ with $\dim S_i \leq 1$. In particular,
$A_i$ has a finite stratification into smooth submanifolds of $X$ of
dimension at most $2$. Similarly, $S_i$ has a finite stratification
into smooth submanifolds of $X$ of dimension at most $1$. According to
Lemma~\ref{lem-2-4}, the normal bundle to $Z_i$ in $X$ is trivial. In
view of the transversality theorem, there exists a $2$-dimensional
compact smooth submanifold $M_i$ of $X$ such that $M_i \cap A_j =
\varnothing$ for $j = 1, 2$, and
\begin{equation*}
M_i \cup \Reg(A_i) = \partial P_i,
\end{equation*}
where $P_i$ is a compact smooth manifold with boundary $\partial P_i$,
embedded in $X$ with trivial normal bundle and satisfying $P_i \cap A_i
= \Reg(A_i)$. We can choose the $M_i$ so that
\begin{equation*}
M_1 \cap M_2 = \varnothing.
\end{equation*}
By Lemma~\ref{lem-2-1}, the smooth submanifold $M_i$ admits an algebraic
approximation in $X \setminus S_i$. Thus, there exists a small smooth
isotopy transforming $M_i$ onto a nonsingular Zariski closed subvariety
$Z'_i$ of $X \setminus S_i$ with $Z'_i \cap A_j = \varnothing$ for $ j =
1, 2$. We can assume that
\begin{equation*}
Z'_1 \cap Z'_2 = \varnothing.
\end{equation*}
If $A'_i$ is the Zariski closure of $Z'_i$ in $X$, then $\Reg(A'_i) =
Z'_i$ and $A'_i \setminus Z'_i \subseteq S_i$. In particular,
\begin{equation*}
\Reg(A'_1 \cup A'_2) = Z'_1 \cup Z'_2.
\end{equation*}
Furthermore, $\llbracket Z'_i \rrbracket_X = \llbracket Z_i
\rrbracket_X$ if $Z'_i$ is suitably oriented. Consequently,
\begin{equation*}
v_1 + v_2 = \llbracket Z_1 \rrbracket_X + \llbracket Z_2 \rrbracket_X =
\llbracket Z'_1 \cup Z'_2 \rrbracket_X
\end{equation*}
is an A-distinguished homology class in $H_2(X; \Z)$, as required.

\begin{assertion}\label{a2}
For each homology class $v$ in $H_2(X; \Z)$, the homology class $2v$ is
of the form $2v = \llbracket V \rrbracket_X$, where $V$ is a nonsingular
Zariski closed subvariety of $X$ that is compact and oriented as a
smooth manifold. In particular, $2v$ is A-distinguished.
\end{assertion}

We have $v = \llbracket M \rrbracket_X$, where $M$ is a $2$-dimensional
compact oriented smooth submanifold of X, cf.~\cite{bib12, bib20} or
\cite[p.~294, Theorem~7.37]{bib19}. By Lemma~\ref{lem-2-4}, the normal bundle to $M$ in
$X$ is trivial. Hence, there exists a $2$-dimensional compact smooth
submanifold $M'$ of $X$ such that $M'$ is isotopic to $M$, $M \cap M' =
\varnothing$, and
\begin{equation*}
M \cup M' = \partial P,
\end{equation*}
where $P$ is a compact smooth manifold with boundary $\partial P$,
embedded in $X$ with trivial normal bundle. If $M'$ is suitably
oriented, then $\llbracket M' \rrbracket_X = \llbracket M \rrbracket_X$.
Furthermore, by Lemma~\ref{lem-2-1} (with $Z = \varnothing$), the smooth
submanifold $M \cup M'$ of $X$ admits an algebraic approximation in $X$.
In particular, $M \cup M'$ is isotopic to a nonsingular Zariski closed
subvariety $V$ of $X$. If $V$ is suitably oriented, then
\begin{equation*}
\llbracket V \rrbracket_X = \llbracket M \cup M' \rrbracket_X = 2
\llbracket M \rrbracket_X = 2v,
\end{equation*}
which proves Assertion~\ref{a2}.

If condition (\ref{lem-2-5-b}) holds, then $u$ can be expressed as
\begin{equation*}
u = w_1 + \cdots + w_r + 2w,
\end{equation*}
where $w_k$ and $w$ are homology classes in $H_2(X; \Z)$, and each $w_k$
is A-distinguished, ${1 \leq k \leq r}$. Thus, in view of
Assertions~\ref{a1} and~\ref{a2}, condition (\ref{lem-2-5-a}) is
satisfied. On the other hand, it is obvious that (\ref{lem-2-5-a})
implies (\ref{lem-2-5-b}).
\end{proof}
\stepcounter{assertionLetter}

\begin{proposition}\label{prop-2-6}
Let $X$ be a nonsingular real algebraic variety of dimension at least
$7$ and let $M$ be a $2$-dimensional compact orientable smooth
submanifold of $X$. Assume that $X$ is a spin manifold. If the homology
class $[M]_X$ belongs to $A_2(X)$, then $M$ admits a weak algebraic
approximation in $X$.
\end{proposition}

\begin{proof}
Endowing $M$ with an orientation, we get $\rho(\llbracket M
\rrbracket_X) = [M]_X$. Now assume that $[M]_X$ belongs to $A_2(X)$.
According to Lemma~\ref{lem-2-5}, the homology class $\llbracket M
\rrbracket_X$ is A-distinguished. Hence
\begin{equation*}
\llbracket M \rrbracket_X = \llbracket Z \rrbracket_X,
\end{equation*}
where $Z$ is a $2$-dimensional nonsingular Zariski locally closed
subvariety of $X$ that is compact and oriented as a smooth manifold.
Moving $M$ by a small smooth isotopy, we can assume that
\begin{equation*}
M \cap Z = \varnothing.
\end{equation*}
The inclusion maps $i \colon M \hookrightarrow X$ and $j \colon Z
\hookrightarrow X$ represent the same class in the second oriented
bordism group $\Omega_2(X)$ of $X$. Indeed, this claim holds since the
canonical, Steenrod--Thom, homomorphism
\begin{equation*}
\Omega_2(X) \to H_2(X; \Z)
\end{equation*}
is an isomorphism, cf.~\cite[p.~75, lines~9,~10]{bib20} or \cite[p.~294, Theorem~7.37]{bib19}. Consequently, there exists
a continuous map $F \colon B \to X$, where $B$ is a compact orientable
smooth manifold with boundary ${\partial B = M \cup Z}$, while $F|_M = i$
and $F|_Z = j$. Since $\dim B = 3$ and $\dim X \geq 7$, we can assume
that $F$ is a smooth embedding. In particular, $Q \coloneqq F(B)
\subseteq X$ is a compact orientable smooth submanifold with boundary
\begin{equation*}
\partial Q = M \cup Z.
\end{equation*}
According to Lemma~\ref{lem-2-4}, the normal bundle to $Q$ in $X$ is
trivial. Hence, Lemma~\ref{lem-2-2} implies that $M$ admits a weak
algebraic approximation in $X$.
\end{proof}

For any $n$-dimensional compact smooth manifold $N$ and any integer $p$,
let
\begin{equation*}
D_N \colon H^p(N; \Z/2) \to H_{n-p}(N; \Z/2)
\end{equation*}
denote the Poincar\'e duality isomorphism.

\begin{lemma}\label{lem-2-7}
Let $X$ be a compact nonsingular real algebraic variety of dimension
$p+k$, where $p \geq 1$ and $k \geq 0$. Let $f \colon X \to \SB^p$ be a
nice continuous rational map. Then $D_X(f^*(s_p)) = [Z]_X$, where $Z$ is
a $k$-dimensional compact nonsingular Zariski locally closed subvariety
of $X$ with trivial normal bundle. In particular, if $X$ is orientable,
then $D_X(f^*(s_p))$ belongs to $A_k(X)$.
\end{lemma}

\begin{proof}
Since $f(P(f))$ is a proper compact subset of $\SB^p$, it follows from
Sard's theorem that the regular map $f|_{X \setminus P(f)} \colon X
\setminus P(f) \to \SB^p$ is transverse to some point $y$ in $\SB^p
\setminus f(P(f))$. Hence $Z \coloneqq f^{-1}(y)$ is a compact
nonsingular Zariski closed subvariety of $X \setminus P(f)$. It is well
known that $D_X(f^*(s_p)) = [Z]_X$, cf.~\cite[Proposition~2.15]{bib8}.
Obviously, the normal bundle to $Z$ in $X$ is trivial. If $X$ is
orientable, then so is $Z$. The proof is complete.
\end{proof}

\begin{proof}[Proof of Theorem~\ref{th-1-2}]
Obviously, (\ref{th-1-2-a}) implies (\ref{th-1-2-b}), while according to
Lemma~\ref{lem-2-7}, (\ref{th-1-2-b}) implies (\ref{th-1-2-c}). It
remains to prove that (\ref{th-1-2-c}) implies (\ref{th-1-2-a}). Assume
that (\ref{th-1-2-c}) is satisfied. We can assume without loss of
generality that $h$ is a smooth map. By Sard's theorem, $h$ is
transverse to some point $y$ in $\SB^p$. Then $M \coloneqq h^{-1}(y)$ is
a $2$-dimensional compact orientable smooth submanifold of $X$
satisfying $D_X(h^*(s_p)) = [M]_X$, cf.~\cite[Proposition~2.15]{bib8}. In
particular, the homology class $[M]_X$ belongs to $A_2(X)$. Hence,
according to Proposition~\ref{prop-2-6}, the submanifold $M$ admits a
weak algebraic approximation in X, which implies that $h$ can be
approximated by nice continuous rational maps,
cf.~\cite[Theorem~1.2]{bib17}. In other words, (\ref{th-1-2-a}) holds.
\end{proof}

\begin{proof}[Proof of Corollary~\ref{cor-1-3}]
By \cite{bib12, bib20} or \cite[p.~294, Theorem~7.37]{bib19}, every
homology class~in $H_2(X; \Z)$ is of the form $\llbracket M
\rrbracket_X$, where $M$ is a $2$-dimensional compact oriented smooth
submanifold of $X$. According to Lemma~\ref{lem-2-4}, the normal bundle
to $M$ in $X$ is trivial, which implies that ${\llbracket M \rrbracket_X
= \rho(\llbracket M \rrbracket_X) = D_X (h^*(s_p))}$ for some smooth map
$h \colon X \to \SB^p$, cf. \cite[Th\'eor\`eme~II.2]{bib20}. In view of
Lemma~\ref{lem-2-7}, $D_X(h^*(s_p))$ belongs to $A_2(X)$, provided that
$h$ is homotopic to a nice continuous rational map. Consequently,
(\ref{cor-1-3-b}) implies (\ref{cor-1-3-c}). According to
Theorem~\ref{th-1-2}, (\ref{cor-1-3-c}) implies (\ref{cor-1-3-a}).
Obviously, (\ref{cor-1-3-a}) implies (\ref{cor-1-3-b}).
\end{proof}

For any real algebraic variety $X$, let $\NF_k(X)$ denote the $k$th
unoriented boridsm group of~$X$. A bordism class in $\NF_k(X)$ is said
to be \emph{algebraic} if it can be represented by a regular map from a
$k$-dimensional compact nonsingular real algebraic variety into $X$,
cf. \cite{bib1,biba}.

\begin{lemma}\label{lem-2-8}
Let $X$ be a compact nonsingular real algebraic variety and let $k$ be a
nonnegative integer. Assume that
\begin{equation*}
\Halg_i(X; \Z/2) = H_i(X; \Z/2)
\end{equation*}
for every integer $i$ such that $0 \leq i \leq k$ and
$\NF_{k-i}(\textrm{point}) \neq 0$. Then each bordism class in
$\NF_k(X)$ is algebraic.
\end{lemma}

\begin{proof}
It suffices to repeat the argument used in the proof of Lemma~2.7.1 in
\cite{biba}.
\end{proof}

\begin{proposition}\label{prop-2-9}
Let $X$ be a compact nonsingular real algebraic variety of dimension
$n$. Let $k$ and $d$ be nonnegative integers satisfying $2k + 1 \leq n$
and $k+1 \leq d$. Assume that
\begin{equation*}
\Halg_i (X; \Z/2) = H_i(X; \Z/2)
\end{equation*}
for $0 \leq i \leq k$. Then any $k$-dimensional compact smooth
submanifold of $X \times \SB^d$ is smoothly isotopic to a nonsingular
Zariski locally closed subvariety of $X \times \SB^d$.
\end{proposition}

\begin{proof}
Let $M$ be a $k$-dimensional compact smooth submanifold of $X \times
\SB^d$ and let $f = (f_1, f_2) \colon M \hookrightarrow X \times \SB^d$
be the inclusion map. Since $2k +1 \leq n$, the map $f_1 \colon M \to X$
is homotopic to a smooth embedding $g_1 \colon M \to X$, cf.
\cite[p.~55, Theorem~2.13]{bib11}. The assumtpion $k+1 \leq d$ implies
that the map $f_2 \colon M \to \SB^d$ is homotopic to a constant map
$g_2 \colon M \to \SB^d$. By construction, the map $g = (g_1, g_2)
\colon M \to X \times \SB^d$ is a smooth embedding homotopic to $f$.
Since $2k+2 \leq n+d$, the maps $f$ and $g$ are isotopic, cf.
\cite[Theorem~6]{bib21} or \cite[p.~183, Exercise~11]{bib11}.
Furthermore, $g(M) = N \times \{y_0\}$, where $N = f_1(M)$ and $\{y_0\}
= g_2(M)$. In particular, the smooth submanifolds $M$ and $N \times
\{y_0\}$ of $X \times \SB^d$ are isotopic. By Lemma~\ref{lem-2-8}, the
unoriented bordism class of the inclusion map $N \hookrightarrow X$ is
algebraic. Consequently, since $\R^d$ is biregularly isomorphic to
$\SB^d$ with one point removed, it follows from \cite[Theorem~F]{bib1}
that the smooth submanifold $N \times \{y_0\}$ is isotopic to a
nonsingular Zariski locally closed subvariety $Z$ of $X \times \SB^d$.
Hence $M$ is isotopic to $Z$, which completes the proof.
\end{proof}

\begin{proof}[Proof of Theorem~\ref{th-1-5}]
It suffices to prove that each smooth map $h \colon X \times \SB^d \to
\SB^p$ can be approximated in $\C(X \times \SB^d, \SB^p)$ by nice
continuous rational maps. By Sard's theorem, $h$ is transverse to some
point $y$ in $\SB^p$. Then $M \coloneqq h^{-1}(y)$ is a compact smooth
submanifold of $X \times \SB^d$ with trivial normal bundle. Either $M =
\varnothing$ or $\dim M = n+d-p$. By Proposition~\ref{prop-2-9}, the
submanifold $M$ is isotopic to a nonsingular Zariski localy closed
subvariety $Z$ of $X \times \SB^d$. It follows that
\begin{equation*}
(M \times \{0\}) \cup (Z \times \{1\}) = \partial B,
\end{equation*}
where $B$ is a compact smooth manifold with boundary $\partial B$,
embedded in $X \times \SB^p \times \R$ with trivial normal bundle. In
view of Proposition~\ref{prop-2-3}, $M$ admits a weak algebraic
approximation in $X \times \SB^d$, which in turn implies that $h$ can be
aproximated by nice continuous rational maps, cf.
\cite[Theorem~1.2]{bib17}.
\end{proof}

\begin{proof}[Proof of Theorem~\ref{th-1-7}]
Each homology class in $\rho(H_2(X; \Z))$ is of the form $[M]_X$ for
some $2$-dimensional compact orientable smooth submanifold $M$ of $X$,
cf. \cite{bib12, bib20} or \cite[p.~294, Theorem~7.37]{bib19}. By the
K\"unneth formula, the group $\rho(H_2(X \times \SB^2; \Z/2))$ is
generated by homology classes of the form $[\{x\} \times \SB^2]_{X
\times \SB^2}$ and $[M \times \{y\}]_{X \times \SB^2}$, where $x \in X$
and $y \in \SB^2$. According to Lemma~\ref{lem-2-8}, the unoriented
bordism class of the inclusion map $M \hookrightarrow X$ is algebraic
(note that $\NF_1(\textrm{point}) = 0$). Consequently, since $\R^2$ is
biregularly isomorphic to $\SB^2$ with one point removed, it follows
from \cite[Theorem~F]{bib1} that the smooth submanifold $M \times \{y\}$
of $X \times \SB^2$ is isotopic to a nonsingular Zariski locally closed
subvariety $Z$ of $X \times \SB^2$. In particular, $[M \times \{y\}]_{X
\times \SB^2} = [Z]_{X \times \SB^2}$. Hence
\begin{equation*}
\rho(H_2(X \times \SB^2; \Z)) = A_2(X \times \SB^2),
\end{equation*}
which in view of Corollary~\ref{cor-1-3} completes the proof.
\end{proof}

We conclude this section by proving the last theorem announced in
Section~\ref{sec-1}. The proof does not depend on the results developed
above.

\begin{proof}[Proof of Theorem~\ref{th-1-8}]
Let $U_p$ be a nonempty open subset of $\SB^{p+1}$ for which the
restriction $\sigma_{U_p} \colon \sigma_p^{-1}(U_p) \to U_p$ of
$\sigma_p$ is a smooth diffeomorphism. We can assume that $h$ is a
smooth map. By Sard's theorem, the smooth map $h \times \one \colon X
\times \SB^1 \to \SB^p \times \SB^1$ is transverse to some point $(y_0,
v_0)$ in $\sigma_p^{-1}(U_p)$. In particular, $M \coloneqq h^{-1}(y_0)$
is a compact smooth submanifold of $X$. If $z_0 = \sigma_p(y_0, v_0)$,
then
\begin{equation*}
(\sigma_p \circ (h \times \one) )^{-1} (z_0) = M \times \{v_0\}
\subseteq X \times \SB^1.
\end{equation*}
Assume that (\ref{th-1-8-b}) holds. According to
\cite[Theorem~2.4]{bib16}, there exists a nonsingular Zariski locally
closed subvariety $Z$ of $X \times \SB^1$ such that
\begin{equation*}
(M \times \{v_0\} \times \{0\}) \cup (Z \times \{1\}) = \partial P,
\end{equation*}
where $P$ is a compact smooth manifold with boundary $\partial P$,
embedded in $X \times \SB^1 \times \R$ with trivial normal bundle. If $F
\colon P \to X$ is the restriction of the canonical projection from $X
\times \SB^1 \times \R$ onto $X$, then $F(x, v_0, 0) = x$ for all $x$ in
$M$, and the restriction $F|_{Z \times \{1\}}$ is a regular map.
Consequently, the unoriented bordism class of the inclusion map $M
\hookrightarrow X$ is algebraic, and hence $M \times \{0\}$ admits a
weak algebraic approximation in $X \times \R$,
cf.~\cite[Theorem~F]{bib1}. Since $\R$ is biregularly isomorphic to
$\SB^1$ with one point removed, it follows that $M \times \{v_0\}$
admits a weak algebraic approximation in $X \times \SB^1$. Thus, in view
of \cite[Theorem~1.2]{bib17}, the continuous map $\sigma_p \circ (h
\times \one)$ can be approximated by nice continuous rational maps. In
other words, (\ref{cor-1-3-b}) implies (\ref{cor-1-3-a}). It is obvious
that (\ref{cor-1-3-a}) implies (\ref{cor-1-3-b}).
\end{proof}

\section{Algebraic approximation of smooth submanifolds}\label{sec-3}
Let $X$ be a nonsingular real algebraic variety. A hard problem is to
find a characterization of these compact smooth submanifolds $M$ of $X$
which admit an algebraic approximation in $X$. A complete solution is
known only if $\codim_X M = 1$ or $(\dim X, \dim M) = (3, 1)$,
cf.~\cite[Theorem~14.4.11]{bib2} and~\cite{bib7}. Very little is known
in other cases. As demonstrated in \cite{bib1}, the problem of algebraic
approximation is more subtle than that of weak algebraic approximation.
In this section, by modifying slightly Propositions~\ref{prop-2-3}
and~\ref{prop-2-6}, we obtain results on algebraic approximation.

\begin{proposition}\label{prop-3-1}
Let $X$ be a nonsingular real algebraic variety and let $M$ be a compact
smooth submanifold of $X$. Assume that there exists a nonsingular
Zariski closed subvariety $Z$ of $X$ such that
\begin{equation*}
(M \times \{0\}) \cup (Z \times \{1\}) = \partial B,
\end{equation*}
where $B$ is a compact smooth manifold with boundary $\partial B$,
embedded in $X \times \R$ with trivial normal bundle. If $2 \dim M + 3
\leq \dim X$, then $M$ admits an algebraic approximation in $X$.
\end{proposition}

\begin{proof}
Arguing as in the proof of Proposition~\ref{prop-2-3}, we can assume that $M
\cap Z = \varnothing$ and
\begin{equation*}
M \cup Z = \partial Q,
\end{equation*}
where $Q$ is a compact smooth manifold with boundary $\partial Q$,
embedded in $X$ with trivial normal bundle. Hence, according to
Lemma~\ref{lem-2-1}, $M$ admits an algebraic approximation in $X$.
\end{proof}

It is now convenient to introduce some notation. Let $X$ be a
nonsingular real algebraic variety. Denote by $B_k(X)$ the subgroup of
$H_k(X; \Z/2)$ generated by all homology classes of the form $[Z]_X$,
where $Z$ is a $k$-dimensional nonsingular Zariski closed subvariety of
$X$ that is compact and orientable as a smooth manifold. Obviously,
\begin{equation*}
B_k(X) \subseteq A_k(X).
\end{equation*}
We say that a homology class $u$ in $H_k(X; \Z)$ is
\emph{B-distinguished} if it is of the form
\begin{equation*}
u = \llbracket Z \rrbracket_X,
\end{equation*}
where $Z$ is as above and endowed with an orientation.

The following is a counterpart of Lemma~\ref{lem-2-5}.

\begin{lemma}\label{lem-3-2}
Let $X$ be a nonsingular real algebraic variety of dimension at least 5.
Assume that $X$ is a spin manifold. For a homology class $u$ in $H_2(X;
\Z)$, the following conditions are equivalent:
\begin{conditions}
\item\label{lem-3-2-a} $u$ is B-distinguished.
\item\label{lem-3-2-b} $\rho(u)$ belongs to $B_2(X)$.
\end{conditions}
\end{lemma}

\begin{proof}
We begin with the following two observations.

\begin{assertion}\label{b1}
If $v_1$ and $v_2$ are B-distinguished homology classes in $H_2(X;
\Z)$, then their sum $v_1 + v_2$ is B-distinguished too.
\end{assertion}

\begin{assertion}\label{b2}
For each homology class $v$ in $H_2(X; \Z)$, the homology class $2v$ is
B-distinguished.
\end{assertion}

The proof of Assertion~\ref{b1} is completely analogous (but simpler) to
that of Assertion~\ref{a1}, while Assertion~\ref{b2} is equivalent to
Assertion~\ref{a2} in the proof of Lemma~\ref{lem-2-5}.

If condition (\ref{lem-3-2-b}) holds, then $u$ can be expressed as
\begin{equation*}
u = w_1 + \cdots + w_r + 2w,
\end{equation*}
where $w_k$ and $w$ are homology classes in $H_2(X; \Z)$, and each $w_k$
is B-distinguished, ${1 \leq k \leq r}$. Thus, in view of
Assertions~\ref{b1} and~\ref{b2}, condition (\ref{lem-3-2-a}) is
satisfied. It is obvious that (\ref{lem-3-2-a}) implies
(\ref{lem-3-2-b}).
\end{proof}

\begin{theorem}\label{th-3-3}
Let $X$ be a nonsingular real algebraic variety of dimension at least
$7$ and let $M$ be a $2$-dimensional compact orientable smooth
submanifold of $X$. Assume that $X$ is a spin manifold. If the homology
class $[M]_X$ belongs to $B_2(X)$, then $M$ admits an algebraic
approximation in $X$.
\end{theorem}

\begin{proof}
Endowing $M$ with an orientation, we get $\rho(\llbracket M
\rrbracket_X) = [M]_X$. Now assume that $[M]_X$ belongs to $B_2(X)$.
According to Lemma~\ref{lem-3-2}, the homology class $\llbracket M
\rrbracket_X$ is B-distinguished. Hence
\begin{equation*}
\llbracket M \rrbracket_X = \llbracket Z \rrbracket_X,
\end{equation*}
where $Z$ is a $2$-dimensional nonsingular Zariski closed subvariety of
$X$ that is compact and oriented as a smooth manifold. Arguing as in the
proof of Proposition~\ref{prop-2-6}, we can assume that $M \cap Z =
\varnothing$ and
\begin{equation*}
M \cup Z = \partial Q,
\end{equation*}
where $Q$ is a compact smooth manifold with boundary $\partial Q$,
embedded in $X$ with trivial normal bundle. Hence, according to
Lemma~\ref{lem-2-1}, $M$ admits an algebraic approximation in $X$.
\end{proof}

It is not known whether the assumptions in Theorem~\ref{th-3-3}
can be relaxed. They certainly cannot be relaxed too much. Indeed, for
any integers $n$ and $k$ satisfying $n-k \geq 2$ and $k \geq 3$, there
exist an $n$-dimensional compact nonsingular real algebraic variety $X$
and a $k$-dimensional compact smooth submanifold $M$ of $X$ such that
$[M]_X = 0$ in $H_k(X; \Z/2)$ and $M$ does not admit an algebraic
approximation in $X$, cf.~\cite[Proposition~1.2]{bib7}.

As a consequence of Theorem~\ref{th-3-3}, we obtain the following.

\begin{example}\label{ex-3-4}
Let $X = C_1 \times \cdots \times C_n$, where each $C_i$ is a compact
connected nonsingular real algebraic curve, $1 \leq i \leq n$. If $n \geq
7$, then each $2$-dimensional compact orientable smooth submanifold of
$X$ admits an algebraic approximation in $X$. Indeed, $B_2(X) = H_2(X;
\Z/2)$ and hence the assertion follows from Theorem~\ref{th-3-3}.
\end{example}

\cleardoublepage
\phantomsection
\addcontentsline{toc}{section}{\refname}
\nocite{*}
\printbibliography

\end{document}